\documentclass[12pt,a4paper]{amsart}
\usepackage{geometry}                
\usepackage{graphicx}
\usepackage{xcolor}
\usepackage{amssymb,MnSymbol,accents}
\usepackage{epstopdf}
\usepackage{tikz-cd}
\newcommand{\livf}{\accentset{\leftarrow}}

\newtheorem*{theorem*}{Theorem}
\newtheorem{theorem}{Theorem}

\theoremstyle{remark}
\newtheorem{remark}[theorem]{Remark}
\newtheorem{example}[theorem]{Example}

\newcommand{\I}{\operatorname{i}}

\renewcommand{\Re}{\operatorname{Re}}

\newcommand{\Span}{\operatorname{span}}  

\newcommand{\C}{\mathbb{C}}

\newcommand{\frk}[1]{\mathfrak{#1}}

\title{CR embeddings of nilpotent Lie groups}
\author{M. G. Cowling}
\address{School of Mathematics and Statistics\\
University of New South Wales\\
UNSW Sydney NSW 2052\\
Australia}
\email{m.cowling@unsw.edu.au}
\author{M. Ganji}
\address{School of Science and Technology\\
University of New England\\
Armidale NSW 2351\\
Australia}
\email{mganjia2@une.edu.au}
\author{A. Ottazzi}
\address{School of Mathematics and Statistics\\
University of New South Wales\\
UNSW Sydney NSW 2052\\
Australia}
\email{a.ottazzi@unsw.edu.au}
\author{G. Schmalz}
\address{School of Science and Technology \\
University of New England \\
Armidale NSW 2351 \\
Australia}
\email{schmalz@une.edu.au}
\thanks{The first and third-named authors were supported by the Australian Research Council grant DP220100285}

\begin{document}
\begin{abstract}
In this note we show that a connected, simply connected nilpotent Lie group with an integrable left-invariant complex structure on a generating and suitably complemented subbundle of the tangent bundle admits a CR embedding in complex space  defined by polynomials. We also show that a similar conclusion holds on suitable quotients of nilpotent Lie groups. Our results extend the CR embeddings constructed by Naruki in 1970. In particular, our generalisation to quotients allows us to see a class of Levi degenerate CR manifolds as quotients of nilpotent Lie groups.
\end{abstract}

\maketitle
\section{Introduction and preliminaries}

We deal throughout with connected, simply connected nilpotent Lie groups $G$ with integrable left-invariant complex structures on certain subbundles of the tangent bundle.

We suppose that $\frk{h}$ is a \emph{horizontal subspace} of the Lie algebra $\frk{g}$ of $G$ with an almost complex structure $J$.
By this, we mean that $\frk{h}$ generates $\frk{g}$ and that there is an ideal $\frk{n}$ in $\frk{g}$ such that
\[
\frk{g}= \frk{h} \oplus \frk{n}.
\]
The Lie algebra $\frk{g}$ and corresponding group are called stratified if 
$$
\mathfrak g = \sum_{j=1}^s \mathfrak g_j,
$$
a vector space direct sum with $[\mathfrak g_1,\mathfrak g_j]=\mathfrak g_{j+1}$ for every $1\leq j <s$.
In many interesting examples, 
the Lie algebra $\mathfrak g$ is stratified and $\mathfrak h = \mathfrak g_1$.
We further suppose that $\frk{h}$ carries an almost complex structure $J$, which implies that ${\rm dim}(\frk{h})$ is even.


The exponential mapping $\exp: \frk{g} \to G$ is a bijection.
We identify $\frk{g}$ with the tangent space to $G$ at the identity $e$ in the usual way,
and write $\livf{X}$ 
for the left-invariant 
vector fields that coincide with $X$ at $e$,  that is,
\[
\livf{X} f(g) = \frac{d}{dt} f(g \exp(tX)) |_{t=0}.
\]

%

Left-translations allow us to transport the subspace $\frk{h}$ of $\frk{g}$ to every point of $G$, and this defines a subbundle $HG$ of the tangent bundle $TG$, with a transported action of $J$ that induces an almost complex structure on $HG$.

We now discuss \emph{integrability}.
The general vector fields of type $(0,1)$ associated with this almost complex structure on $HG$ are, by definition, those generated as linear combinations of $a^X(\livf{X} + \I J\livf{X})$, where  $a^X$ is a complex-valued function on $G$ and $X\in \frk{h}$.
Observe that
\[
\begin{aligned}
&[a^X(\livf{X} + \I J\livf{X}), b^Y(\livf{Y} + \I J \livf{Y})] \\
&\qquad= a^X(\livf{X} b^Y + \I J\livf{X} b^Y)  (\livf{Y} + \I J\livf{Y})
 -  b^Y(\livf{Y} a^X + \I J\livf{Y} a^X)(\livf{X} + \I J\livf{X}) \\
 &\qquad\qquad\qquad + a^X b^Y[\livf{X}+ \I J\livf{X}, \livf{Y} + \I J\livf{Y}] ,
\end{aligned}
\]
and this is a vector field of type $(0,1)$ associated with the almost complex structure on $HG$ if and only if the subspace $\{ X + \I JX : X \in \frk{h} \}$ of the complexified algebra $\frk{h}_{\C}$ is in fact a subalgebra, that is, it is closed under commutation.
We assume this, or  equivalently, that
\begin{equation}\label{eq:integrable}
[{X}, J{Y}]  + [ J{X}, {Y}] = J ([ {X}, {Y}] - [J{X}, J{Y}])   \in \frk{h}
\end{equation}
whenever $X,Y\in \frk{h}$.
We write $\frk{h}^{01}$ for $\{ X + \I JX : X \in \frk{h} \}$, and note that $\frk{h}^{01}$ determines $J$ and vice versa.
A structure on a nilpotent Lie group of the form described is said to be an integrable left-invariant horizontal CR structure of type $(\frac{1}{2}{\rm dim}\mathfrak h,{\rm dim}\mathfrak n)=(n,k)$. If $\mathfrak g$ is stratified and $\mathfrak h = \mathfrak g_1$, then we say that the CR structure is homogeneous.
We fix a basis of $\mathfrak g$ as follows. We choose a basis $\{X_1,\dots,X_{2n}\}$ of $\mathfrak h$ and a basis $\{X_{2n+1},\dots,X_{2n+k}\}$ of $\mathfrak n$.  
We recall that an embedding of $M$ into $N$ of CR manifolds is said to be a CR embedding when its differential maps $(0,1)$ vector fields on $M$ to  $(0,1)$ vector fields on $N$.
\color{black}

In light of the example of the Heisenberg group as the boundary of a Siegel domain, the Lie theoretic generalisations of this example by Murakami \cite{Mur}, and the more general examples of Beloshapka \cite{Bel}, we ask whether this CR manifold admits a CR embedding into the complex space $\C^{n+k}$ as the edge of a wedge in a complex domain of the form
\begin{equation}\label{eq:domain}
\{ (x + \I y, u+\I v) \in \C^{n+k}: v_i > q_i(x,y,u) \},
\end{equation}
where each $q_i$ is a real polynomial. Here $x$ and $y$ are the coordinates corresponding to the horizontal space $\mathfrak h$, where we  defined the almost complex structure $J$. 

 If the CR structure is homogeneous, Naruki showed the path to answer our question in \cite{Naruki} and Gregorovi\v{c} made it explicit in a recent paper \cite{Gre}. However, the approach of Naruki and Gregorovi\v{c}  produce embeddings described by possibly complex polynomials, making the interpretation as the edge of a wedge of a domain as in \eqref{eq:domain} no longer possible.
%
 In fact, Naruki's method does not allow us to distinguish the cases where an embedding of the type \eqref{eq:domain} exists.
If the CR structure on $G$ is not homogeneous, we will generalise Naruki's strategy. 


\begin{theorem}\label{thm-embedding}
Every nilpotent Lie group with an integrable left-invariant horizontal CR structure of type $(n,k)$ admits a CR embedding $\iota$ into $\C^{n+k}$ of the form 
$$
({\bf x} + \I {\bf y}, {\bf t}) \mapsto ({\bf x} + \I {\bf y} + {\bf p}({\bf x},{\bf y}), {\bf t}+  {\bf q}({\bf x},{\bf y},{\bf t})),
$$
where ${\bf x},{\bf y}\in {\mathbb R}^n, {\bf u}\in {\mathbb R}^k$, ${\bf p}({\bf x},{\bf y})$ is a vector of $n$ (possibly complex valued) polynomials, and ${\bf q}({\bf x},{\bf y},{\bf t})$ is a vector of $k$ (possibly complex valued) polynomials.
\end{theorem}
We will give an example to illustrate our method and to point out that interesting examples may arise in the non-homogeneous case. In particular, $\frk{h}^{01}$ needs no be abelian. 
Motivated by the example of the Heisenberg group, our second result is about a group structure on the target space of these embeddings.
\begin{theorem}\label{thm-groupstructure}
Let $G$ be a nilpotent group with an integrable left-invariant horizontal CR structure of type $(n,k)$.
 Then there exists a product on $\C^{n+k}$ such that the embedding $\iota$ is a homomorphism from $G$ into $\C^{n+k}$. In other words, $\iota$ realises $G$ as a subgroup of $\C^{n+k}$.
\end{theorem}

Under the hypotheses of the theorem above,  suppose further that $P\subset G$ is a subgroup such that its Lie algebra satisfies $\mathfrak p\subset \frk{n}$. 
In this case the CR-structure on $G$ projects to a CR-structure on the quotient $P\backslash G$ of the same dimension, because the intersection $\mathfrak p\cap \mathfrak h=\{0\}$. 
\begin{theorem}\label{thm-quotients}
Every CR-embedding of $G$ into $\C^{n+k}$ projects to an embedding of $P\backslash G$ into $\C^{n+k'}$, with $k'=k-{\rm dim}P$.
\end{theorem}
We will show that interesting examples arise in the quotient context. In particular, we will see that a class of Levi-degenerate finite type hypersurfaces in $\mathbb C^2$ (see, e.g., \cite{Kolar}) arise by embedding suitable quotients of stratified groups.

Before proving our statements, we need some more notation and a few preliminary results.  
We consider the complexification $\frk{g}_{\mathbb C}$ of $\frk{g}$ and the corresponding Lie group $G_{\mathbb C}$. Notice that we may write $\frk{g}_{\mathbb C}$ as the vector space direct sum $\frk{g}_{\mathbb C}=\frk{h}^{10}\oplus  {\frk{h}^{01}} \oplus \frk{n}_{\mathbb C}$, where $\frk{h}^{10} = \overline{\frk{h}^{01}}$ and $\frk{n}_{\mathbb C}$ is the complexification of $\frk{n}$. 
The space $\frk{m} = {\frk{h}^{10}} \oplus \frk{n}_{\mathbb C}$ is a complex subalgebra of $\frk{g}_{\mathbb C}$.  The exponential map $ \exp$ from $\frk{g}_{\mathbb C}$ is surjective onto $G_{\mathbb C}$, and therefore is surjective from all subalgebras of $\frk{g}_{\mathbb C}$ onto a corresponding connected and simply connected subgroup of $G_{\mathbb C}$. Unless stated otherwise, we will identify $G_{\mathbb C}$ or its subgroups with  $\frk{g}_{\mathbb C}$ or a suitable subalgebra. Namely, $\forall p\in G_{\mathbb C}$, 
$$p=(x_1,\dots,x_{2n},x_{2n+1},\dots,x_{2n+k}) = \exp \left(\sum_{j=1}^{2n+k}x_jX_j\right).$$

We define the Lie groups $H^{01} = \exp (\frk{h}^{01})$ and $M = \exp(\frk{m})\simeq{\mathbb C}^{n+k}$.
Consider the canonical projection
$$
\pi: G^{\mathbb C} \to G^{\mathbb C}/H^{01}.
$$
Notice that the restriction of $\pi$ to $M$ is a holomorphic diffeomorphism onto $G^{\mathbb C}/H^{01}$ and let $\Psi$ be the inverse of this diffeomorphism. 
Then $\Phi:=\Psi\circ\pi: G^{\mathbb C}\to M$ is surjective. Denote by $\Phi_R^{-1}$  a right inverse of $\Phi$. We define a product $\ast$ on $M$ as follows. For every $m,m'\in M$, 
\begin{equation}\label{product-ast}
m\ast m' := \Phi\left(\Phi_R^{-1}(m)\Phi_R^{-1}(m')\right).
\end{equation}
In other words, $\ast$ is the product that makes $\Phi_R^{-1}$ a homomorphism.

\section{Theorems \ref{thm-embedding} and \ref{thm-groupstructure}}
In this section, we prove Theorems \ref{thm-embedding} and \ref{thm-groupstructure}. The proof of Theorem~\ref{thm-embedding} is a reformulation of Naruki's argument~\cite{Naruki}. However, our approach is made simpler by the product $\ast$ on $M$.
We will also provide an example at the end of the section where $\frk{h}^{01}$ is not commutative, and compute an explicit embedding in this case.
We recall the statement of our first theorem for the reader's convenience.
%
\begin{theorem*}
Every nilpotent Lie group with an integrable left-invariant horizontal CR structure of type $(n,k)$ admits a CR embedding $\iota$ into $\C^{n+k}$ of the form 
$$
({\bf x} + \I {\bf y}, {\bf t}) \mapsto ({\bf x} + \I {\bf y} + {\bf p}({\bf x},{\bf y}), {\bf t}+  {\bf q}({\bf x},{\bf y},{\bf t})),
$$
where ${\bf x},{\bf y}\in {\mathbb R}^n, {\bf u}\in {\mathbb R}^k$, ${\bf p}({\bf x},{\bf y})$ is a vector of $n$ (possibly complex valued) polynomials, and ${\bf q}({\bf x},{\bf y},{\bf t})$ is a vector of $k$ (possibly complex valued) polynomials.
\end{theorem*}

\begin{proof}
Since $G\cap H^{01} = \{e\}$, the canonical projection $\pi: G^{\mathbb C} \to G^{\mathbb C}/H^{01}$
is injective when restricted to $G$ and so $\pi(G)$ is a submanifold of the complex manifold $G^{\mathbb C}/H^{01}$.  Let $\iota := \Phi\Big|_{G}:G\to M$, where $\Phi$ was defined at the end of the previous section. We show that $\iota$ is a homomorphism into $(M,\ast)$. Indeed, for every $g,g'\in G$,
\begin{equation}\label{iota-hom}
\iota(gg')=\iota(\iota^{-1}\iota(g)\iota^{-1}\iota(g'))=\iota(g)\ast\iota(g').
\end{equation}
We show that $\iota$ is a CR embedding.  
To this end, it suffices\footnote{Please Gerd confirm that this is indeed the case} to show that the differential $\iota_*$ satisfies $\iota_*(\livf{JX})=i\iota_*(\livf{X})$ for every $X\in \mathfrak h$. First, notice that 
\begin{align*}
\iota (\exp tJX) &=\iota\left(\exp\left(\frac{t}{2}(JX+iJ^2X)+\frac{t}{2}(JX-iJ^2X)\right)\right)\\
&=\iota\left(\exp\left(\frac{-it}{2}(X+iJX)+\frac{it}{2}(X-iJX)\right)\right)
\end{align*}
and 
\begin{equation*}
\exp\left(\frac{-it}{2}(X+iJX)+\frac{it}{2}(X-iJX)\right)\exp\left(\frac{it}{2}(X+iJX)\right)=\exp\left(\frac{it}{2}(X-iJX)+O(t^2)\right).
\end{equation*}
Similarly,
$$
\iota (\exp tiX) =\iota\left(\exp\left(\frac{it}{2}(X+iJX)+\frac{it}{2}(X-iJX)\right)\right)
$$
and 
\begin{equation*}
\exp\left(\frac{it}{2}(X+iJX)+\frac{it}{2}(X-iJX)\right)\exp\left(\frac{-it}{2}(X+iJX)\right)=\exp\left(\frac{it}{2}(X-iJX)+O(t^2)\right).
\end{equation*}
Hence, for every $X\in \mathfrak h$, every smooth function $f:M\to\mathbb R$ and $p\in G$,
\begin{align*}
\iota_*(\livf{JX}) (f)(\iota(p) )&= \livf{JX}(f\circ \iota)(p)\\
&= \frac{d}{dt} (f\circ \iota)(p \exp(tJX)) |_{t=0}\\
&= \frac{d}{dt} f(\iota(p)\ast\iota(\exp(tJX)) |_{t=0}\\
&=\frac{d}{dt} f\left(\iota(p)\ast\exp\left(\frac{it}{2}(X-iJX)\right) \right) |_{t=0}\\
&=\frac{d}{dt} f(\iota(p)\ast \iota(\exp i tX)) |_{t=0}\\
&=\frac{d}{dt} f(\iota(p\exp i tX)) |_{t=0}\\
&=i\livf{X}(f\circ \iota)(p)\\
&= i\,\iota_*(\livf{X}) (f)(\iota(p) ).
\end{align*}

Next, let $p=\exp(\sum_{j=1}^{2n+k} x_jX_j)$. Since $\frk{g}_{\mathbb C} =\frk{h}^{10}\oplus  {\frk{h}^{01}} \oplus \frk{n}_{\mathbb C}$ and $\frk{g}_{\mathbb C}$ is nilpotent, there exist polynomials $\psi_j$, $j=1,\dots,n+k$, and $\varphi_\ell$, $\ell=1,\dots,n$, in the variables ${\bf x}=(x_1,\dots,x_{2n+k})$ so that
$$
\exp(\sum_{j=1}^{2n+k} x_jX_j) = \exp\left(\sum_{j=1}^n\psi_j({\bf x})(X_j+iJ(X_j))+\sum_{\ell=n}^{n+k}\psi_\ell({\bf x})X_\ell\right)\exp\left(\sum_{j=1}^n\varphi_j({\bf x})(X_j-iJ(X_j))\right).
$$
Then $\Psi\circ \pi(g)=\exp\left(\sum_{j=1}^n\psi_j({\bf x})(X_j+iJ(X_j))+\sum_{\ell=n}^{n+k}\psi_\ell({\bf x})X_\ell\right)$. A careful analysis of the Baker-Campbell-Hausdorff formula and the nilpotency of $G$ show that the polynomials $\psi_j$ have the properties required. For more details, see \cite{Gre}.
\end{proof}

Theorem~\ref{thm-groupstructure} is a Corollary of Theorem~\ref{thm-embedding} and its proof.
\begin{theorem*}
Let $G$ be a nilpotent group with an integrable left-invariant horizontal CR structure of type $(n,k)$.
 Then there exists a product on $\C^{n+k}$ such that the embedding $\iota$ is a homomorphism from $G$ into $\C^{n+k}$. In other words, $\iota$ realises $G$ as a subgroup of $\C^{n+k}$.
\end{theorem*}
\begin{proof}
From~\eqref{iota-hom}, $\iota$ is a homomorphism and therefore it maps $G$ into a subgroup of $(M,\ast)$.
\end{proof}

 \begin{example}
Let $\frk{g}=\Span\{X_1, \dots, X_8\}$, with nontrivial commutators defined by
\[
[X_2,X_3]=X_5,\quad [X_2,X_4]=X_6,\quad [X_3,X_4]=X_7,\quad [X_3,X_7]=X_8,
\]
and let $G$ be the corresponding connected, simply connected nilpotent Lie group.
Choose $\frk{h}=\Span\{X_1, \dots, X_6\}$, $\frk{h}^{01}=\Span\{X_1+\I X_2, X_3+ \I X_4, X_5+ \I X_6\}$, and  $\frk{n}=\Span\{X_7, X_8\}$.
Since $[ X_1+\I X_2, X_3+ \I X_4] = \I(X_5+ \I X_6)$ and $X_5$ and $X_6$ are central, $\frk{h}^{01}$ is  a noncommutative subalgebra of $\frk{g}_{\C}$.
Evidently $\frk{n}$ is an ideal that complements $\frk{h}$.
Let $A=a_1(X_1+\I X_2) +a_2(X_3+\I X_4)+a_3(X_5+\I X_6)$, $B=b_1(X_1-\I X_2) +b_2(X_3-\I X_4)+b_3(X_5-\I X_6)$, and $C=c_1X_7+c_2X_8$, with $a_j,b_k,c_\ell\in \mathbb C$. We define $\Psi\circ \pi=\Phi:G_{\mathbb C}\to M$ by
$$
{p}=\exp(A+B+C) \mapsto \exp(A+B+C) \exp\left(-A-\frac{1}{2}ia_2b_1(X_5+iX_6)\right).
$$
A straightforward computation yields
$$
\Phi(p) = (b_1,b_2,b_3+\frac{i}{2}a_1b_2,c_7-ia_2b_2,c_8+a_2c_7-\frac{i}{3}a_2^2b_2-\frac{i}{6}a_2b_2^2)
$$
The embedding of $G$ is then obtained by restricting $\Phi$ to $G$, i.e., by substituting $a_1=\frac{1}{2}(x_1-ix_2)$, $b_1=\frac{1}{2}(x_1+ix_2)$, $a_2=\frac{1}{2}(x_3-ix_4)$,
$b_2=\frac{1}{2}(x_3+ix_4)$, $a_3=\frac{1}{2}(x_5-ix_6)$, $b_3=\frac{1}{2}(x_5+ix_6)$, $c_1=x_7$, and $c_2=x_8$. The embedding is
\begin{align*}
(x_1,\dots,x_8) &\mapsto \left(\frac{1}{2}(x_1+ix_2), \frac{1}{2}(x_3+ix_4), \frac{1}{2}(x_5+ix_6)+\frac{i}{8}(x_1-ix_2)(x_3+ix_4),\right.\\
&\left.\hskip3cm x_7-\frac{i}{4}(x_3^2+x_4^2),x_8+\frac{1}{4}x_7(x_3-ix_4)
-\frac{1}{48}(x_4+3ix_3)(x_3^2+x_4^2)\right).
\end{align*}

\end{example}

\begin{remark}
In the case where the CR structure is homogeneous, we can write an explicit expression for $\Phi$ and $\Phi_R^{-1}$, which is convenient in calculations. 
We  write every $p\in G_{\mathbb C}$ as
$p=\exp(A+B+C)$, where $A\in \mathfrak h^{01}$, $B\in \mathfrak h^{10}$, and $C\in \mathfrak n_{\mathbb C}$.
Then $\Phi: G_{\mathbb C}\to M$ is defined by
$$\Phi(\exp(A+B+C))=\exp(A+B+C)\exp(-A).$$
The right inverse $\Phi_R^{-1}:M\to G_{\mathbb C}$ can be chosen to be 
$$\Phi_R^{-1}(\exp(B+C))=\exp(A+C)\exp(\bar B),$$
for every $B\in  \mathfrak h^{10}$ and $C\in \mathfrak n_{\mathbb C}$. 

%
%
\end{remark}

\vskip0.2cm
The product $\ast$ on $M$ that we defined in \eqref{product-ast} generalises a well known result for the Heisenberg group, as we show in the following example.
\begin{example}
Let $\mathfrak g={\rm span}\{X_1,X_2,X_3\}$ be the three dimensional Heisenberg algebra. The only non-zero bracket between the elements of the basis is set to be $[X_1,X_2]=X_3$. We choose the left-invariant CR structure defined by $\livf{X}_1 + \I \livf{X}_2$, with $\livf{X}_1$ and $\livf{X}_2$ denoting the left-invariant vector fields on the group $G=\exp(\mathfrak g)$ that coincide with $X_1$ and $X_2$ at the identity. 
Then $\mathfrak h={\rm span}\{X_1,X_2\}$, $\mathfrak h^{01}= {\rm span}\{X_1+iX_2\}$,
$\mathfrak h^{10}= {\rm span}\{X_1-iX_2\}$, and $\mathfrak n={\rm span}\{X_3\}$. After complexification, we see that $\mathfrak m = {\rm span}_{\mathbb C}\{X_1-iX_2,X_3\}\simeq \mathbb C^2$. Let $(z,w),(z',w')\in M\simeq \mathbb C^2$. From~\eqref{product-ast}
and a straightforward computation we obtain that
$$
(z,w)\ast(z',w') = (z+z',w+w'-2i\bar{z}z').
$$ 
This product is holomorphic from the left, in the sense that the mapping $$L_{(z,w)}:(z',w')\mapsto (z,w)\ast(z',w')$$ is holomorphic with respect to $\frac{\partial}{\partial \bar{z}'}$ and $\frac{\partial}{\partial \bar{w}'}$. If $p=\exp(x_1X_1+x_2X_2+x_3X_3)$,  the embedding $\iota : G \to \mathbb C^2$ is 
$$
(x_1,x_2,x_3)\mapsto \left(\frac{1}{2}(x_1+ix_2),x_3-\frac{i}{4}(x_1^2+x_2^2)\right).
$$
\end{example}

\section{Theorem~\ref{thm-quotients}}
Let $P\subset G$ be a subgroup such that its Lie algebra satisfies $\mathfrak p\subset \frk{n}$.
Denote by ${\mathfrak p}_{\mathbb C}$ and $P_{\mathbb C}$ the corresponding complexifications. We prove the following theorem.
\begin{theorem*}
Every CR-embedding of $G$ into $\C^{n+k}$ projects to an embedding of $P\backslash G$ into $\C^{n+k'}$, with $k'=k-{\rm dim}P$.
\end{theorem*}
\begin{proof}

Since $\mathfrak p\cap \mathfrak h=\{0\}$,  the map $\Phi=\Psi\circ \pi$ is injective when restricted to $P_{\mathbb C}$. In fact, $\Phi(P_{\mathbb C})=P_{\mathbb C}$ is a subgroup of $(M,\ast)$ and $p\ast p' = pp'$ for all $p,p'\in 
P_{\mathbb C}$.
Define 
$$
\tilde{\Phi}: P_{\mathbb C} \backslash G_{\mathbb C} \to \Phi(P_{\mathbb C}) \backslash M
$$
by the requirement that the following diagram commutes
\begin{center}
\begin{tikzcd}
    G_{\mathbb C} \arrow{r}{\Phi} \arrow{d}[swap]{\pi_1} & M \arrow{d}{\pi_2} \\
    P_{\mathbb C} \backslash G_{\mathbb C} \arrow{r}[swap]{\tilde{\Phi}} &  \Phi(P_{\mathbb C}) \backslash M
\end{tikzcd}
\end{center}
where $\pi_1$ and $\pi_2$ denote the canonical projections. Then the restriction of $\tilde{\Phi}$ to $P\backslash G$ is a CR embedding.
Indeed, recalling that $\iota_*(\livf{JX})=i\iota_*(\livf{X})$ and since $H^{01}\cap P=\{e\}$, we conclude that
$$
\left(\tilde{\Phi}\Big|_G \right)_* (\pi_1)_* (\livf{JX})= (\pi_2)_*\left({\Phi}\Big|_G \right)_*(\livf{JX}) = (\pi_2)_*\iota_*(\livf{JX})= i(\pi_2)_*\iota_*(\livf{X})
$$

\end{proof}

\begin{example}

Consider the 
$6$-dimensional  Lie algebra 
$\mathfrak{g}={\rm span}\{X_1,\dots,X_6\}$ where the nontrivial brackets are given by
$$
X_3=[X_2,X_1], \quad X_4=[X_3,X_1], \quad X_5=[X_3,X_2],\quad [X_4,X_1]=[X_5,X_2]=8X_6.
$$
Denote by $G$ the connected and simply connected Lie group with Lie algebra $\mathfrak{g}$. 
 We use  exponential coordinates of the second kind on $G$
$$
{\bf x} = (x_1,x_2,x_3,x_4,x_{5},x_6):= \exp\left(\sum_{j=3}^{5}x_jX_j\right)\exp\left(x_1X_1+x_2X_2\right)
\exp\left(x_6 X_6\right).$$
The left-invariant vector fields on $G$ corresponding to $X_1$ and $X_2$ in these coordinates are
$$
\livf{X}_1 =  \partial_{x_1}+ \frac{x_2}{2}\partial_{x_3}-\frac{1}{3}x_1x_2 \partial_{x_4}-\frac{1}{3}x^2_2 \partial_{x_5}+x_2(x_1^2+x_2^2) \partial_{x_6}
$$
and 
$$
\livf{X}_2 =  \partial_{x_2}- \frac{x_1}{2}\partial_{x_3}+\frac{1}{3}x^2_2 \partial_{x_4}+\frac{1}{3}x_1x_2 \partial_{x_5}-x_1(x_1^2+x_2^2) \partial_{x_6}.
$$
Then $\livf{X}_1+i\livf{X}_2$ defines a left-invariant CR-structure on $G$.
In the complexified group $G_{\mathbb C}$, we write a point as $\exp(C)\exp(A+B)$, with $A = a(X_1+iX_2)$, $B=b(X_1-iX_2)$,  $C=\sum_{j=3}^6 c_j X_j$, and $a,b,c_j\in \mathbb C$. A straightforward computation shows that 
\begin{align*}
\Phi\left(\exp(C)\exp(A+B)\right)&= \exp(C)\exp(A+B)\exp(-A)\\
&= \exp\left((c_3+iab)X_3 +(c_4-\frac{i}{3}ab(a+2b))X_4+(c_5+\frac{1}{3}ab(a-2b)X_5\right)\times\\
&\qquad \times \exp(b(X_1-iX_2)) \exp(c_6+4ia^2b^2X_6).
\end{align*}
Then the embedding $\iota$ of $G$  in  $M\simeq \mathbb C^5$ is obtained by substituting in the formula above $a=(x_1-ix_2)/2$, $b=(x_1+ix_2)/2$, and $c_j=x_j$. This yields
\begin{align*}
\iota:(x_1,\dots,x_6)&\mapsto \left(\frac{1}{2}(x_1+ix_2),x_3 + \frac{i}{4}(x_1^2 + x_2^2),x_4 - \frac{i}{24}(x_1^2+x_2^2)(3x_1+ix_2),\right.\\
&\hskip5cm\left. x_5 - \frac{1}{24}(x_1^2+x_2^2)(x_1+3ix_2), x_6+\frac{i}{4}(x_1^2 + x_2^2)^2\right).
\end{align*}

Next, let $\mathfrak{p}={\rm span}\{X_2,\dots,X_5\}$, $P=\exp \mathfrak{p}$ and $P_{\mathbb C}$ its complexification. The quotient $P\backslash G$ is a three dimensional CR manifold where points can be identified with $(x_1,x_2,x_6)\in\mathbb R^3$. Moreover, we see that $\Phi$ restricted to $P_{\mathbb C}$ is the identity. Therefore, we may compute $\tilde{\Phi}$ applied to $P\backslash G$ and obtain that
%
$$
(x_1,x_2,x_6)\mapsto  \left(\frac{1}{2}(x_1-ix_2),x_6-\frac{i}{4}(x_1^2 + x_2^2)^2\right),
$$
which shows that the Levi degenerate manifold 
$$
\Omega= \left\{(z,w)\in {\mathbb C}^2: {\rm Im}w = -\frac{1}{4}(x_1^2+x_2^2)^{2}\right\},
$$
is CR equivalent to a left quotient of a Lie group by a subgroup. Here $z=x_1+ix_2$ and $x_6= \Re w$.
The CR structure on $\Omega$  is defined by $L=({\pi_2})_*(\livf{X}_1)+i({\pi_2})_*(\livf{X}_2)=\tilde{X}_1+i\tilde{X}_2$ with
$$ \tilde{X}_1= \partial_{x_1}+x_2(x_1^2+x_2^2) \partial_{x_6}, $$
$$ \tilde{X}_2= \partial_{x_2} -x_1(x_1^2+x_2^2) \partial_{x_6}. $$
A similar construction can be replicated for the CR manifolds 
$$
\Omega_k= \left\{(z,w)\in {\mathbb C}^2: {\rm Im}w = -\frac{1}{2k}(x_1^2+x_2^2)^{k}\right\},
$$
for all $k\geq 2$. In this general case, we construct $G$ as follows. We observe that $\tilde{X}_1+i\tilde{X}_2$ with
$$ \tilde{X}_1= \partial_{x_1}+x_2(x_1^2+x_2^2)^{k-1} \partial_{x_N}, $$
$$ \tilde{X}_2= \partial_{x_2} -x_1(x_1^2+x_2^2)^{k-1} \partial_{x_N} $$
defines a CR structure on the manifold $\Omega_k$, where $z=x_1+ix_2$ and $\Re w =x_N$. Hence, we define $\mathfrak g$ to be the Lie algebra generated by $\tilde{X}_1$ and $\tilde{X}_2$ in the case that $x_1^2+x_2^2\neq 0$. These examples are constructed more explicitly in~\cite{ChangLOW}.
\end{example}

\section{Afterthoughts}
The\footnote{this section can be deleted altogether or edit at leisure} embedding that we constructed in the proof of Theorem~\ref{thm-embedding} is not the only CR embedding that we can define using exponential coordinates of the first kind.
In fact, in many  examples, the embedding in the statement of Theorem~\ref{thm-embedding} can be constructed in a form where ${\bf p}={\bf 0}$, and ${\bf q}$
is a vector of real valued polynomials, that is, in the form given in \eqref{eq:domain}. In~\cite{CLOW} the authors present an algorithm to compute a CR embedding of the form \eqref{eq:domain} for a stratified group with a homogeneous CR structure,  under the assumption that such an embedding exists. They then show that this does exist for all free Lie algebras with two generators and up to step $8$.
The following example confirms that the algorithm in~\cite{CLOW} and that of  Theorem~\ref{thm-embedding} may have different outcomes.

\begin{example}
Let $\mathfrak g=\Span\{X_1,X_2,X_3,X_4,X_5\}$ be the free nilpotent Lie algebra of step three with two generators. The non-zero brackets are
$$
[X_1,X_2]=X_3,\quad [X_1,X_3]=X_4,\quad [X_3,X_2]=X_5.
$$
We define the CR structure on $G=\exp\mathfrak g$ induced by $L=X_1+iX_2$.
From~\cite{CLOW}, it follows that an embedding of the form~\eqref{eq:domain} exists (with $\mathfrak h=\mathfrak g_1$). However, if we follow the construction of Theorem~\ref{thm-embedding}, we see that $\Phi:G_{\mathbb C}\to \mathbb C^4$ is
\begin{align*}
\exp(A+B+C) &\mapsto \exp\left(A+(c_3-abi)X_3+ (c_4+\frac{ac_3}{2}-\frac{i}{3}a^2b-\frac{i}{6}ab^2)X_4\right.\\
&\hskip6cm+\left. (c_5-\frac{i}{2}ac_3+\frac{1}{6}ab^2-\frac{1}{3}a^2b)X_5\right)
\end{align*}
where $A=a(X_1+iX_2)$, $B=b(X_1-iX_2)$, and $C= \sum_{j=3}^5 c_jX_j$. The embedding is found by taking the restriction of $\Phi$ to $G$. Namely, we apply $\Phi$ to $a=\frac{x_1-\I x_2}{2}$, $b=\frac{x_1+\I x_2}{2}$, and $c_j = x_j$, for $j=3,4,5$. We see immediately that the coefficient of $X_4$, for example, is 
$$x_4+ \frac{1}{4}x_1x_3-\frac{i}{4}x_2x_3 -\frac{i}{24}(x_1-ix_2)(x_1^2+x_2^2)-\frac{i}{48}(x_1+ix_2)(x_1^2+x_2^2),
$$
 which is not of the form \eqref{eq:domain}.
\end{example}
This example shows that, at least in some cases, we can find a better form for the embedding expressed in exponential coordinates. It would be interesting to characterise the classes of nilpotent groups with an integrable left-invariant CR structure for which an embedding of the form \eqref{eq:domain} exists.

\end{document}